\documentclass{amsart}

\usepackage{amssymb}

\usepackage[pagewise, displaymath, mathlines]{lineno}

\makeatletter
\let\my@abstract=\relax
\def\abstract#1{%
  \def\my@abstract{%
    \normalfont\Small
    \list{}{\labelwidth\z@
      \leftmargin3pc \rightmargin\leftmargin
      \listparindent\normalparindent \itemindent\z@
      \parsep\z@ \@plus\p@
      \let\fullwidthdisplay\relax
    }%
    \item[\hskip\labelsep\scshape\abstractname.]%
    #1
  \endlist}}
\def\@setabstracta{%
  \ifx\my@abstract\relax
  \else
    \skip@20\p@ \advance\skip@-\lastskip
    \advance\skip@-\baselineskip \vskip\skip@
  \my@abstract
    \prevdepth\z@ 
  \fi
}
\makeatother

\newtheorem{theorem}{Theorem}[section]
\newtheorem{lemma}[theorem]{Lemma}
\newtheorem{cor}[theorem]{Corollary}

\theoremstyle{definition}

\newtheorem{example}[theorem]{Example}

\newtheorem{rem}[theorem]{Remark}

\newtheorem{prop}[theorem]{Proposition}
\theoremstyle{remark}

\numberwithin{equation}{section}

\usepackage{xypic}
\xyoption{all}

\begin{document}


\title[Witt equivalence of function fields]{Witt equivalence of function fields over 
global fields}

\author{Pawe\l \ G\l adki and Murray Marshall}

\address{Institute of Mathematics,
University of Silesia, \newline \indent
ul. Bankowa 14, 40-007 Katowice, Poland \newline \indent
\and \newline \indent
Department of Computer Science,
AGH University of Science and Technology, \newline \indent
al. Mickiewicza 30, 30-059 Krak\'ow, Poland}
\email{pawel.gladki@us.edu.pl}

\address{Department of Mathematics and Statistics,
University of Saskatchewan, \newline \indent
106 Wiggins Rd., Saskatoon, SK S7N 5E6, Canada}
\email{marshall@math.usask.ca}



\thanks{The research of the second author was supported in part by NSERC of Canada.}

\abstract{ Witt equivalent fields can be understood to be fields having the same symmetric bilinear form theory. Witt equivalence of finite fields, local fields and global fields is well understood. Witt equivalence of function fields of curves defined over archimedean local fields is also well understood.
In the present paper, Witt equivalence of general function fields over global fields is studied. It is proved that for any two such fields $K,L$, any Witt equivalence $K \sim L$ induces a cannonical bijection $v \leftrightarrow w$ between Abhyankar valuations $v$ on $K$ having residue field not finite of characteristic $2$ and  Abhyankar valuations $w$ on $L$ having residue field not finite of characteristic $2$. The main tool used in the proof is a method for constructing valuations
due to Arason, Elman and Jacob \cite{aej}. The method of proof does not extend to
non-Abhyankar valuations. The result is applied to study Witt equivalence of function fields
over number fields. It is proved, for example, that if $k$, $\ell$ are number fields and $k(x_1,\dots,x_n) \sim \ell(x_1,\dots,x_n)$, $n\ge 1$, then $k \sim \ell$ and
the  $2$-ranks of the ideal class groups of $k$ and $\ell$ are equal.
}

\maketitle

\section{introduction}

Let $K$ be a field.
Denote by $W(K)$ the Witt ring of (non-degenerate) symmetric bilinear forms over $K$; see \cite{l}, \cite{M1} or \cite{w} for the definition in case $\operatorname{char}(K) \ne 2$ and \cite{KR}, \cite{krw} or \cite{mh} for the definition in the general case. Denote by $Q(K)$ the quadratic hyperfield of $K$; roughly speaking this is the same thing as the quadratic form scheme of $K$ \cite{kss} \cite{l};
see Section 3 for the definition.
We say two fields $K,L$ are \it Witt equivalent, \rm denoted $K \sim L$, if $Q(K) \cong Q(L)$ as hyperfields, equivalently, if $W(K) \cong W(L)$ as rings; see Proposition \ref{harrison} below. Witt equivalent fields can be understood as fields having the same symmetric bilinear form theory.

Witt equivalence of finite fields and local fields is well understood.   Witt equivalence of global fields is considered in \cite{C}, \cite{Petal}, \cite{Sz1}, \cite{Sz2}, \cite{Sz0}. Witt equivalence of function fields of curves defined over local and global fields is considered in \cite{gh}, \cite{kop1}, \cite{kop2}. (Note, however, that there
is a serious error in the proof of Theorem 1.3 in \cite{kop1}, in the proof of (1.3.1) $\Rightarrow$ (1.3.2).)

It is well-known that any hyperfield isomorphism $\alpha : Q(K) \rightarrow Q(L)$ carries orderings of $K$ to orderings of $L$ in the sense that if $P \subseteq K^*$ is the positive cone of an ordering  of $K$ then
\begin{linenomath}\[Q = \{ s \in L^* :  \overline{s} = \alpha(\overline{t}) \text{ for some } t\in P\}\]\end{linenomath}
 is the positive cone of an ordering of $L$. Here, $\overline{x}$ denotes the image of $x$ under the canonical map $K^* \rightarrow K^*/K^{*2}$. This correspondence can also be deduced from the fact that orderings on $K$ correspond to ring homomorphisms from $W(K)$ to $\mathbb{Z}$.

It is natural to wonder if a similar result holds for valuations, i.e., if the valuations of a field $K$ can be detected by looking at the quadratic hyperfield $Q(K)$.
At this level of generality the result is 
false. E.g., $\mathbb{C} \sim \mathbb{F}_2$ and $\mathbb{C}((x)) \sim \mathbb{F}_5$. In each of these examples, the first field has lots of non-trivial valuations, but the second field has only the trivial valuation. At the same time, there is a detection procedure which works for certain sorts of fields. E.g., if $K,L$ are global fields of characteristic $\ne 2$, then any hyperfield isomorphism $\alpha : Q(K) \rightarrow Q(L)$ induces in a cannonical way a bijection $v \leftrightarrow w$ between valuations $v$ of $K$ and valuations $w$ of $L$; see \cite{C}, \cite{Petal}, \cite{Sz1}, \cite{Sz2}, \cite{Sz0}. The main tool  for setting up this bijection is a method of constructing valuations described in \cite{aej}, which is based, in turn, on earlier constructions, of a similar sort, described in \cite{j0} and \cite{wa}.

In the present paper we extend the above-mentioned result for global fields, proving that if $K,L$ are function fields over global fields then any hyperfield isomorphism $\alpha : Q(K) \rightarrow Q(L)$ induces in a canonical way a bijection $v \leftrightarrow w$ between Abhyankar valuations $v$ of $K$ having residue field not finite of characteristic $2$ and Abhyankar valuations $w$ of $L$ having residue field not finite of characteristic $2$; see Theorem \ref{main theorem 2}.

Our results are applied to study Witt equivalence of function fields over number fields; see
Corollary \ref{rational points}, Theorem \ref{genus zero case} and Corollary \ref{infinitely many}.
It is proved, for example, that if $k(x_1,\dots,x_n) \sim \ell(x_1,\dots,x_n)$, where $n\ge 1$ and $k$ and $\ell$ are number fields, then $k \sim \ell$ and the $2$-ranks of the ideal class groups of $k$ and $\ell$ are equal.

In Sections 2 and 3 we recall basic terminology which is used throughout the paper.
In Section 4 we establish basic connections between quadratic hyperfields and valuations. In Section 5 we apply the result in \cite{aej} to understand the behavior of valuations under Witt equivalence; see Theorem \ref{main lemma}. In Section 6 we recall the terminology of function fields, global fields and Abhyankar valuations, and we introduce the idea of nominal transcendence degree.

The main new results in the paper are found in Sections 5,7 and 8.

The authors would like to express their thanks to the annonymous referee for his careful reading and useful improvements that made the presentation more comprehensible.

\section{hyperfields}

A hyperfield is an object like a field, but where the addition is allowed to be multivalued. Hyperfields
were introduced by Krasner \cite{kr0}, \cite{kr}, in connection with his work on valuations.
Hyperfields were also introduced independently in \cite{M2} where they were called multifields.

A \it hyperfield \rm is a system $(H, +, \cdot, -, 0, 1)$ where $H$ is a set, $+$ is a multivalued binary operation on $H$, i.e., a function from $H\times H$ to the set of all subsets of $H$, $\cdot$ is a binary operation on $H$, $- : H \rightarrow H$ is a function, and $0,1$ are elements of $H$ such that

\smallskip

I. $(H,+,-,0)$ is a canonical hypergroup, terminology as in Mittas \cite{mi}, i.e.,
\smallskip

(1) $c\in a+b$ $\Rightarrow$ $a \in c+(-b)$,

(2) $a \in b+0$ iff $a=b$,

(3) $(a+b)+c = a+(b+c)$, and

(4) $a+b= b+a$; and

\smallskip

II. $(H,\cdot, 1)$ is a commutative monoid, i.e., $(ab)c=a(bc)$, $ab=ba$,
and $a1=a$ for all $a, b, c \in A$; and

\smallskip

III. $a0 = 0$ for all $a\in H$; and

\smallskip
IV. $a(b+c) \subseteq ab+ac$; and

\smallskip

V. $1\ne 0$ and every non-zero element has a multiplicative inverse.

\smallskip

Hyperfields form a category. A \it morphism \rm from $H_1$ to $H_2$, where $H_1$, $H_2$ are hyperfields, is a function $\alpha : H_1 \rightarrow H_2$ which satisfies $\alpha(a+b) \subseteq \alpha(a)+\alpha(b)$, $\alpha(ab) = \alpha(a)\alpha(b)$, $\alpha(-a)=-\alpha(a)$, $\alpha(0)=0$, $\alpha(1)=1$.

Here are some elementary consequences of the hyperfield axioms: (i) $-0=0$ (ii) $-(-a)=a$ (iii) $a+b \ne \emptyset$ 
(iv) $a(-b)=-(ab)$ (v) $(-a)(-b)=ab$.

Every field is a hyperfield.
The simplest non-trivial examples of hyperfields are the quotient hyperfields.
If $T$ is a subgroup of $H^*$, where $H$ is a field or hyperfield, the \it quotient hyperfield \rm $H/_m T = (H/_m T, +,\cdot, -, 0,1)$ is defined as follows:
$H/_m T$ is the set of equivalence classes with respect to the equivalence relation $\sim$ on $H$ defined by $a\sim b$ iff $as = bt$ for some $s,t\in T$. The operations on $H/_m T$ are the obvious ones induced by the corresponding operations on $H$: Denote by $\overline{a}$ the equivalence class of $a$. Then $\overline{a} \in \overline{b}+\overline{c}$ iff $as \in bt+cu$ for some $s,t,u \in T$, $\overline{a}\overline{b} = \overline{ab}$, $-\overline{a} = \overline{-a}$. Also, $0 = \overline{0}$, and $1 = \overline{1}$. The group of non-zero elements of $H/_m T$ is $H^*/T$. The subscript $m$ here is used to indicate that $H/_m T$ is a quotient modulo a multiplicative subgroup $T$ and was introduced in \cite{M2}: although we call $H/_m T$ a quotient, its construction really resembles more that of a localisation, and the authors believe that denoting it simply by $H/T$ might be somewhat misleading.

The hyperfield associated to an ordered abelian group $\Gamma  := (\Gamma, \cdot, 1, \le)$ is $\Gamma \cup \{ 0\} := (\Gamma \cup \{ 0\}, +, \cdot, -, 0,1)$,
where
\begin{linenomath}\[a+ b := \begin{cases} b \text{ if } a<b \\ a \text{ if } b<a\\ [0,a] \text{ if } a = b \end{cases},\]\end{linenomath}
$a\cdot 0 = 0\cdot a := 0$ and $-a :=a$.  Convention: $0<a$ for all $a \in \Gamma$.

A valuation on a field $K$ is just a morphism $v : K \rightarrow \Gamma \cup \{ 0\}$, for some ordered abelian group $\Gamma:= (\Gamma, \cdot, 1, \le)$. If $\Gamma$ is the value group of $v$, i.e., if $v$ is surjective, then $v$ induces an isomorphism $\overline{v} : K/_m U \rightarrow \Gamma \cup \{ 0 \}$, where $U$ is the unit group of $v$.\footnote{The foregoing example notwithstanding, in what follows we will always use the more standard
additive notation for valuations, i.e., a valuation is a function $v : K \rightarrow \Gamma\cup \{ \infty\}$, for some ordered abelian group $\Gamma := (\Gamma, +, 0, \ge)$.}

See \cite{mas} for an example of a hyperfield which is not realizable as a quotient hyperfield of a field.

We are mostly interested in one special example of a quotient hyperfield, namely the hyperfield $K/_m {K^*}^2$, for a fixed field $K$, and its particular connection to symmetric bilinear forms over $K$. Observe that, for a field $K$, and for $z, a, b \in K$ the following equivalence holds true:
$$z = ax^2 + by^2 \mbox{ for some } x, y \in K^*  \mbox{ if and only if } \overline{z} \in \overline{a} + \overline{b} \mbox{ in } K/_m {K^*}^2.$$
It turns out that, in fact, a slightly more general equivalence holds true, at least when $K \neq \mathbb{F}_2, \mathbb{F}_3$, $\operatorname{char} (K) \neq 2$ (see Proposition \ref{properties} below for details), namely:
$$z = ax^2 + by^2 \mbox{ for some } x, y \in K  \mbox{ if and only if } \overline{z} \in \overline{a} + \overline{b} \mbox{ in } K/_m {K^*}^2.$$
This equivalence fails to hold without these additional assumptions. Here it is necessary to modify the definition of addition in $K/_m {K^*}^2$, defining $\overline{a} + \overline{b}$, for $a, b \neq 0$, "by hand". Fortunately enough, this can be also done more conceptually, by defining a new addition on any given hyperfield.

If $H = (H, +,\cdot, -, 0,1)$ is a hyperfield, the \it prime addition \rm on $H$ is defined by
\begin{linenomath}\[a+'b = \begin{cases} a+b &\text{ if one of } a,b \text{ is zero } \\ a+b\cup \{ a,b\} &\text{ if } a \ne 0, \ b\ne 0,\ b \ne -a \\ H &\text{ if } a \ne 0, \ b\ne 0,\ b = -a \end{cases}.\]\end{linenomath}

In the next section we use the following result:

\begin{prop} For any hyperfield $H:=(H, +,\cdot, -, 0,1)$, $H':=(H, +',\cdot, -, 0,1)$ is also a hyperfield.
\end{prop}

\begin{proof} We make use of the fact that $a+b \subseteq a+'b$. I (1) Suppose $c\in a+'b$. If $c\in a+b$ then $a\in c+(-b) \subseteq c+'(-b)$. Otherwise, $a,b\ne 0$ and $c=a$ or $c=b$ or $a=-b$. In each of these cases, $a \in c+'(-b)$ is clear. (2) Since $b+'0 = b+0$ this is clear. (3) As explained in \cite[Lemma 1.3]{M2}, it suffices to show $(a+'b)+'c \subseteq a+'(b+'c)$, i.e., if $x\in y+'c$ for some $y\in a+'b$ then $x\in a+'z$ for some $z\in b+'c$. If $x\in y+c$ and $y\in a+b$ this is clear. Otherwise either $y,c\ne 0$ and ($x=y$ or $x=c$ or $y=-c$) or $a,b\ne 0$ and ($y=a$ or $y=b$ or $a=-b$). In the first case, if $x=y$ take $z=b$ if $b\ne 0$ and $z=c$ if $b=0$; if $x=c$ take $z=c$; if $y=-c$ take $z=-a$ if $a\ne 0$ and $z=x$ if $a=0$. In the second case, if $y=a$ take $z=c$ if $c\ne 0$ and $z=b$ if $c=0$; if $y=b$ take $z=x$ if $x\ne 0$ and $z=-a$ if $x=0$; if $a=-b$ take $z=b$. (4) is clear. II, III and V are clear. IV Suppose $x\in b+'c$. If $x\in b+c$ then $ax \in a(b+c) \subseteq ab+ac \subseteq ab+'ac$. Otherwise, $b,c\ne 0$ and $x=b$ or $x=c$ or $b=-c$. In each of these cases $ax \in ab+'ac$ is clear.
\end{proof}

We refer to $H'$ as the \it prime \rm of the hyperfield $H$.
Observe that if $T$ is a subgroup of $H^*$ then $H'/_mT = (H/_mT)'$.

\section{Quadratic hyperfields and Witt equivalence}

Let $K$ be a field.
The \it quadratic hyperfield \rm of $K$, denoted $Q(K)$, is defined to be the prime of the hyperfield $K/_m K^{*2}$.\footnote{This is the same object referred to
in \cite[page 458]{M2}. Roughly speaking, it is the quadratic form scheme of $K$, terminology as in \cite{kss} or \cite{l}, with zero adjoined.} Note that $Q(K)^* = K^*/K^{*2}$.

\begin{prop}\label{properties} Assume  $\overline{a} \in Q(K)^*$. Then

(1) $\overline{a}^2=\overline{1}$.

(2) If $\overline{a} \ne -\overline{1}$ then $\overline{1}+\overline{a}$ is a subgroup of $Q(K)^*$.\footnote{ If $G = (G,-1,V)$ is an (abstract) quadratic form scheme, terminology as in \cite{kss}, then $H = (H, +,\cdot, -, 0,1)$,
where $H := G \cup \{ 0 \}$, \begin{linenomath}\[a+b := \begin{cases} a \text{ if } b=0 \\ b \text{ if } a = 0 \\ a\cdot V(ab) \text{ if } a,b\ne 0, b\ne -a \\ H \text{ if } a,b \ne 0, b=-a \end{cases},\]\end{linenomath} $a\cdot 0 = 0\cdot a := 0$ and $-a := (-1)\cdot a$, is hyperfield satisfying (1) and (2) of Proposition \ref{properties}, i.e., for all $a \in H^*$ (1) $a^2=1$ and (2) if $a\ne -1$ then $1+a$ is a subgroup of $H^*$. Conversely, every hyperfield $H$ satisfying (1) and (2) arises in this way, from some unique quadratic form scheme $G$. See \cite[Theorem 1.4]{kss} for some equivalent descriptions of quadratic form schemes. The question of whether every quadratic form scheme is realized as the quadratic form scheme of a field appears to be
still open.}

(3) If $K \ne \mathbb{F}_3, \mathbb{F}_5$ and $\operatorname{char}(K) \ne 2$ then $Q(K)=K/_m K^{*2}$.
\end{prop}

\begin{proof} (1) $a \in K^*$ so $a^2\in K^{*2}$. It follows that $\overline{a}^2=\overline{a^2}=\overline{1}$. (2) If $\overline{0} \in \overline{1}+\overline{a}$, then $\overline{a} \in \overline{0}+(-\overline{1})$, so $\overline{a} = -\overline{1}$, which contradicts our assumption. This proves $\overline{1}+\overline{a} \subseteq K^*/K^{*2}$. Clearly $\overline{1} \in \overline{1}+\overline{a}$. Each $\overline{b} \in K^*/K^{*2}$ satisfies $\overline{b}\, \overline{b} = \overline{b^2} = \overline{1}$, so is its own inverse. Closure of $\overline{1}+\overline{a}$ under multiplication follows from the standard identity \begin{linenomath}\[(x_1^2+ay_1^2)(x_2^2+ay_2^2)= (x_1x_2-ay_1y_2)^2+a(x_1y_2+x_2y_1)^2.\]\end{linenomath}
(3) It suffices to show $\forall$ $b\in K$, $b \in aK^{*2}-aK^{*2}$. Scaling, we are reduced to the case $a=1$. If $b\ne \pm1$, the identity $b=(\frac{b+1}{2})^2-(\frac{b-1}{2})^2$ shows that $b \in K^{*2}-K^{*2}$. Thus we are reduced to showing $\pm 1 \in K^{*2}-K^{*2}$. Scaling, we are reduced further to showing $1\in K^{*2}-K^{*2}$. Since  $K \ne \mathbb{F}_3, \mathbb{F}_5$, and $\operatorname{char}K\ne 2$, $|K^*| \ge 6$, so there exists $b\in K^*$, $b^2 \ne \pm 1$. Then $b^2 = (\frac{b^2+1}{2})^2-(\frac{b^2-1}{2})^2$, so, dividing by $b^2$, $1 \in K^{*2}-K^{*2}$.
\end{proof}

The interest in $Q(K)$ stems from its connection to
symmetric bilinear forms over $K$. One is mainly interested in the characteristic $\ne 2$ case.
In this case, symmetric bilinear forms and quadratic forms are the same thing.

Denote by $W(K)$ the Witt ring of non-degenerate symmetric bilinear forms over $K$; see \cite{l}, \cite{M1} or \cite{w} for the definition in case $\operatorname{char}(K) \ne 2$ and \cite{KR}, \cite{krw} or \cite{mh} for the definition in the general case.

A (non-degenerate diagonal) \it binary form \rm over $K$ is just an ordered pair $\langle \overline{a},\overline{b}\rangle$, $\overline{a}, \overline{b} \in K^*/K^{*2}$. The \it value set \rm of such a form, denoted by  $D_K\langle \overline{a},\overline{b} \rangle$, is the set of non-zero elements of $\overline{a}+\overline{b}$, i.e., $D_K\langle \overline{a},\overline{b}\rangle$ is the image under $K^* \rightarrow K^*/K^{*2}$ of the subset $D_K\langle a,b\rangle$ of $K^*$ defined by \begin{linenomath}\[D_K\langle a,b\rangle := \begin{cases} K^* \text{ if } -ab \in K^{*2} \\ \{ z\in K^* : z = ax^2+by^2, x,y \in K\} \text{ otherwise}\end{cases}.\]\end{linenomath}
Two binary forms $\langle \overline{a},\overline{b}\rangle$ and $\langle \overline{c},\overline{d}\rangle$
are considered to be \it equivalent, \rm denoted $\langle \overline{a},\overline{b}\rangle \approx \langle \overline{c},\overline{d}\rangle$, if $\overline{c}\in D_K\langle \overline{a},\overline{b}\rangle$ and $\overline{a}\overline{b} = \overline{c}\overline{d}$.

In terms of generators and relations,
$W(K)$ is the integral group ring $\mathbb{Z}[K^*/K^{*2}]$ factored by the ideal generated by $[1]+[-1]$ and all elements \begin{linenomath}\[[a]+[b]-[c]-[d] \text{ such that } \overline{a},\overline{b},\overline{c},\overline{d} \in K^*/K^{*2}, \ \langle \overline{a},\overline{b}\rangle \approx \langle \overline{c},\overline{d}\rangle.\]\end{linenomath}
See \cite[Theorem 1.16 (iv) and Corollary 1.17]{krw} for the proof. Here, $[x]$ denotes the image of $\overline{x}$ under the canonical embedding $K^*/K^{*2} \hookrightarrow \mathbb{Z}[K^*/K^{*2}]$.

A hyperfield isomorphism $\alpha : Q(K) \rightarrow Q(L)$, where $K,L$ are fields, can be viewed as a group isomorphism $\alpha : K^*/K^{*2} \rightarrow L^*/L^{*2}$ such that $\alpha(-\overline{1}) = -\overline{1}$ and \begin{linenomath}\[\alpha(D_K\langle \overline{a},\overline{b}\rangle) = D_L\langle \alpha(\overline{a}),\alpha(\overline{b})\rangle \text{ for all } \overline{a},\overline{b} \in K^*/K^{*2},\]\end{linenomath} or, equivalently, as a group isomorphism  $\alpha : K^*/K^{*2} \rightarrow L^*/L^{*2}$ which induces a ring isomorphism between $W(K)$ and $W(L)$.
We say two fields $K$ and $L$ are \it Witt equivalent, \rm denoted $K\sim L$, to indicate that
$Q(K)$ and $Q(L)$ are isomorphic as hyperfields. For completeness and clarity we record the following:

\begin{prop}\label{harrison} $K\sim L$ iff $W(K)$ and  $W(L)$ are isomorphic as rings.
\end{prop}

\begin{proof} See \cite{h} for the characteristic $\ne 2$ case. As remarked in \cite{bm}, the 
Hauptsatz in \cite{ap} holds for all characteristics. The general case follows from this fact; see \cite[Proposition 4.6]{M1}.
\end{proof}

For fields 
of characteristic $\ne 2$, Witt equivalence is also characterized in terms of Galois groups; see \cite[Theorem 3.8]{ms}.

It is well-known that the Witt ring of a field $K$ encodes the theory of symmetric bilinear forms over $K$. Witt equivalent fields can be understood as fields having the same symmetric bilinear form theory. The quadratic hyperfield $Q(K)$ encodes exactly the same information as the Witt ring $W(K)$.
At the same time, it is a much simpler and easier object to deal with.

Hyperfields provide a first-order axiomatization of the algebraic theory of quadratic forms. Although other first-order descriptions have been already known for some time (see \cite{dm2000} and \cite{M-etc}), it seems that the theory of hyperfields is the most natural and the most easily understood. All the results presented in what follows can be "translated" to the traditional notion of Witt rings, and, as of today, the authors are not familiar with any results in the algebraic theory of quadratic forms that can be proven with the use of hyperfields, but can not be proven without them. Still, the authors believe that hyperfields make the exposition easier to read and to understand.

\section{Quadratic hyperfields and valuations}

Let $H_1, H_2$ be hyperfields. Each morphism $\iota: H_1 \rightarrow H_2$ induces a morphism $\overline{\iota} : H_1/_m \Delta \rightarrow H_2$ where $\Delta: = \{ x\in H_1^* : \iota(x)=1\}$. The morphism $\iota$ is said to be a \it quotient morphism \rm if $\overline{\iota}$ is an isomorphism, equivalently, if $\iota$ is surjective, and $\iota(c) \in \iota(a)+\iota(b)$ iff $cs\in at+bu$ for some $s,t,u \in \Delta$.  A morphism $\iota : H_1 \rightarrow H_2$ is said to be a \it group extension \rm if $\iota$ is injective,
every $x\in H_2^* \backslash \iota(H_1^*)$ is \it rigid \rm in the sense that $1+x \subseteq \{ 1,x\}$,\footnote{We are interested here in the case where the groups $H_1^*, H_2^*$ have exponent $2$. In this situation, $1+x \subseteq \{ 1,x\}$ $\Leftrightarrow$ $1+x = \{1,x\}$.} and
 $y \in H_1$, $y\ne -1$ $\Rightarrow$ $\iota(1+y)=1+\iota(y)$.

We assume now that $K$ is a field. For a valuation $v$ on $K$, $\Gamma_v$ denotes the value group, $A_v$ denotes the valuation ring, $M_v$ the maximal ideal, $U_v$ the unit group, and $K_v$ the residue field. $\pi = \pi_v: A_v \rightarrow K_v$ denotes the canonical homomorphism, i.e., $\pi(a) = a+M_v$. We say $v$ is \it discrete rank one \rm if $\Gamma_v = \mathbb{Z}$. See \cite{e}, \cite{ep}, \cite{r} for background material on valuations.

We will be interested in the subgroup $T= (1+M_v)K^{*2}$ of $K^*$.

\begin{prop} \label{small detail} Suppose $v$ is non-trivial and $T = (1+M_v)K^{*2}$. Then:

(1)  $T\cup xT \subseteq T+xT$ for all $x\in K^*$;

(2) $T-T=K$;

(3) The map $Q(K) \rightarrow K/_mT$ defined by $\overline{x} \mapsto xT$  is a quotient morphism.
\end{prop}

\begin{proof} (1) Pick $p \in K^*$ so that $v(p^2)> \max\{ v(x), -v(x)\}$. Since we are assuming $v$ is non-trivial this is always possible. Then $t = 1+p^2x \in T$, so $1 = \frac{1}{t}(1+p^2x) \in T+xT$, and, similarly, $p^2+x = x(1+\frac{p^2}{x})\in xT$, so $x\in T+xT$. (2) Suppose $y\in K$. Pick $p \in K^*$ so that $v(p^2)< v(y)$. Then $y = p^2(1+ \frac{y}{p^2})-p^2 \in T-T$. (3) In view of (1) and (2), $K/_mT = (K/_mT)'$, so this is clear.
\end{proof}

Propositions \ref{group extension} and \ref{completion} below are variants of old results of Springer \cite{sp1}, \cite{sp2} couched in the language of quadratic hyperfields.
Consider the canonical group isomorphism $\alpha : U_vK^{*2}/(1+M_v)K^{*2} \rightarrow K_v^*/K_v^{*2}$ induced by $x \in U_v \mapsto \pi(x) \in K_v^*$. Define $\iota : Q(K_v) \rightarrow K/_m T$ by $\iota(0) = 0$ and $\iota(a) = \alpha^{-1}(a)$ for $a \in K_v^*/K_v^{*2}$.

\begin{prop} \label{group extension} Suppose $v$ is non-trivial and $T = (1+M_v)K^{*2}$. Then:

(1) $\iota$ is a morphism;

(2) $\iota$ is a group extension.
\end{prop}

Note: The cokernel of the group embedding $\alpha^{-1} : K_v^*/K_v^{*2} \rightarrow K^*/T$ is equal to $K^*/U_vK^{*2} \cong \Gamma_v/2\Gamma_v$. For this reason we sometimes say that $K/_mT$ is a \it group extension of $Q(K_v)$ by the group \rm $\Gamma_v/2\Gamma_v$.

\begin{proof} (1) $\iota(ab) = \iota(a)\iota(b)$, $\iota(-a)= -\iota(a)$, $\iota(0)=0$ and $\iota(1)=1$ are clear. It remains to show $\iota(a+b) \subseteq \iota(a)+\iota(b)$. This is clear if one of $a,b$ is zero, so we can assume $a,b \ne 0$. Scaling, we are reduced to showing $\iota(1+a) \subseteq 1+\iota(a)$ for all $a \in K_v^*/K_v^{*2}$. Represent $a$ by an element $\pi(x)$, $x\in U_v$. Suppose $\pi(y) = \pi(p)^2+\pi(q)^2\pi(x)$, $p,q \in U_v$, $y \in A_v$.
Then $y= p^2+q^2x+z$, $v(z)>0$, so $y = p^2(1+\frac{z}{p^2})+q^2x \in T+Tx$. In view of parts (1) and (2) of Proposition \ref{small detail}, this proves (1).

(2) Clearly $\iota$ is injective. Suppose $y= t_1+t_2x$, $t_1,t_2 \in T$, $x\notin U_vK^{*2}$. Then $v(t_1) \ne v(t_2x)$. If $v(t_1) < v(t_2x)$, then $y = t_1(1+\frac{t_2x}{t_1}) \in T$. If $v(t_1) > v(t_2x)$, then $y = t_2x(1+\frac{t_1}{t_2x}) \in Tx$. This proves the rigidity assertion. Suppose now that $y = t_1+t_2x$, $t_1,t_2 \in T$, $x\in U_v$, $\pi(x) \notin -K_v^{*2}$. We want to show $\exists$ $y' \in Ty\cap U_v$ such that $\pi(y') \in K_v^{*2} +K_v^{*2}\pi(x)$ or $\pi(y') \in K_v^{*2}$ or $\pi(y') \in K_v^{*2}\pi(x)$. If $v(y) > \min \{ v(t_1), v(t_2x)\}$ then $x \in -T$, which contradicts $\pi(x) \notin -K_v^{*2}$. Thus $v(y) = \min \{ v(t_1), v(t_2x)\}$.  If $v(t_1)\le v(t_2)$ take $y' = \frac{y}{t_1}$. If $v(t_1)> v(t_2)$ take $y' = \frac{y}{t_2}$.
\end{proof}

\begin{prop} \label{completion} Suppose $v$ is non-trivial, $\operatorname{char}(K_v) \ne 2$,  and $T = (1+M_v)K^{*2}$. Then $K/_m T$ is naturally identified with $Q(\tilde{K}_v)$, where $\tilde{K}_v$ denotes the henselization of $(K,v)$.
\end{prop}

Note: The conclusions of Propositions \ref{small detail}, \ref{group extension} and \ref{completion} also hold when $v$ is trivial, provided $K \ne \mathbb{F}_3, \mathbb{F}_5$ and $\operatorname{char}(K) \ne 2$.

\begin{proof} Denote by $\tilde{v}$ the extension of $v$ to $\tilde{K}_v$. Since $(\tilde{K}_v,\tilde{v})$ is henselian and $\operatorname{char}(K_v) \ne 2$, $1+M_{\tilde{v}} \subseteq \tilde{K}_v^{*2}$. It follows that the embedding $K \hookrightarrow \tilde{K}_v$ induces a group homomorphism $\tau : K^*/T \rightarrow \tilde{K}_v^*/\tilde{K}_v^{*2}$. Since $(\tilde{K}_v, \tilde{v})$ is an immediate extension of $(K,v)$, one sees that $\tau$ is an group isomorphism. The image of $Q(K_v)^*$ in $K^*/T$ under the  group extension $\iota: Q(K_v) \hookrightarrow K/_m T$ is identified via $\tau$ with the image of $Q(K_v)^*$ in $Q(\tilde{K}_v)^*$ under the group extension  $\iota :  Q(K_v) \hookrightarrow Q(\tilde{K}_v)$. The conclusion follows from this.
\end{proof}

If $v$ is discrete rank one, one can replace henselization by completion in Proposition \ref{completion}.
The assumption in Proposition \ref{completion} that $\operatorname{char}(K_v)\ne 2$  is crucial. One says that $v$ is \it dyadic \rm  if $\operatorname{char}(K)=0$, $\operatorname{char}(K_v) = 2$. The structure of $Q(\tilde{K}_v)$ when $v$ is dyadic is complicated; see \cite{l} or \cite{M1} for the case where $K$ is a number field and \cite{j} and \cite{j1} for the case where $K$ is arbitrary.

\begin{rem} Suppose $v,v'$ are valuations on $K$ with $v \preceq v'$, i.e., $v'$ is a coarsening of $v$, i.e., $A_v \subseteq A_{v'}$. Then $M_{v'} \subseteq M_v$ so $(1+M_{v'})K^{*2} \subseteq (1+M_v)K^{*2}$. Denote by $\overline{v}$ the valuation on $K_{v'}$ induced by $v$, i.e., $\overline{v}(\pi_{v'}(a)) = v(a)$, for $a\in U_{v'}$.
Note that $\overline{v}$ and $v$ have the same residue field. See \cite[Chapter C]{r} for background. Assume now that $v,v'$ are non-trivial and that $v'$ is a proper coarsening of $v$.
Then $K/_m (1+M_v)K^{*2}$ is a group extension of the hyperfield $K_{v'}/_m(1+M_{\overline{v}})K_{v'}^{*2}$ in a natural way,
and the following diagram of hyperfields and hyperfield morphisms is commutative:
\begin{linenomath}\begin{equation}\label{d0} \xymatrix{
Q(K) \ar[r] & K/_m(1+M_{v'})K^{*2} \ar[r] & K/_m(1+M_v)K^{*2} \\
& Q(K_{v'}) \ar[u] \ar[r] & K_{v'}/_m (1+M_{\overline{v}})K_{v'}^{*2} \ar[u]\\
& & Q(K_v) \ar[u]
}\end{equation}\end{linenomath}
Here, the horizontal arrows are quotient morphisms and the vertical arrows are group extensions.
\end{rem}

Let $T$ be a subgroup of $K^*$.
We say $x\in K^*$ is \it $T$-rigid \rm if $T+Tx \subseteq T\cup Tx.$
\begin{linenomath}\[B(T) := \{ x \in K^*: \text{ either } x \text{ or } -x \text{ is not } T\text{-rigid}\}.\]\end{linenomath}
Elements of $B(T)$ are said to be \it $T$-basic. \rm
Note that if $x\in K^*$ is $T$-rigid and
$y=tx$, $t\in T$, then $y$ is $T$-rigid. Consequently, $B(T)$ is a union of cosets of $T$.
$-1$ is not $T$-rigid (because $0\in T-T$), so $\pm T \subseteq B(T)$.
We say that $T$ is \it exceptional \rm if $B(T)=\pm T$ and either $-1\in T$ or $T$ is additively closed.

We recall the result of Arason, Elman and Jacob alluded to in the introduction:

\begin{theorem}\label{seeing valuations} Let $T \subseteq K^*$ be a subgroup and $H \subseteq K^*$ be a subgroup containing $B(T)$. Then there exists a subgroup $\hat{H}$ of $K^*$ such that $H \subseteq \hat{H}$ and $(\hat{H}:H) \le 2$ and a valuation $v$ of $K$ such that $1+M_v \subseteq T$ and $U_v \subseteq \hat{H}$. Moreover, $\hat{H}=H$ works, unless $T$ is exceptional.
\end{theorem}

\begin{proof} See \cite[Theorem 2.16]{aej}.
\end{proof}

We will apply Theorem \ref{seeing valuations} to study Witt equivalence of function fields over global fields. We make frequent use of the following:

\begin{prop} \label{remark on basic part} \

(1) $B(K^{*2})$ is a subgroup of $K^*$.

(2) Suppose $T = (1+M_v)K^{*2}$ for some non-trivial valuation $v$ of $K$. Then $B(T) \subseteq U_vK^{*2}$ and \begin{linenomath}\[B(T) = \{ x\in K^* : \overline{x} = \iota(\overline{y}) \text{ for some } y \in B(K_v^{*2})\},\]\end{linenomath} where $\iota : Q(K_v) \hookrightarrow K/_mT$ is the morphism in Proposition \ref{group extension}.  $B(T)$ is a group and the group isomorphism $\iota : K_v^*/K_v^{*2} \rightarrow U_vK^{*2}/T$ induces a group isomorphism $B(K_v^{*2})/K_v^{*2} \rightarrow B(T)/T$. $T$ is exceptional iff $K_v^{*2}$ is exceptional.
\end{prop}

\begin{proof} (1) This is due to L. Berman. See \cite[Theorem 5.18]{M1} for the proof. (2)  The fact that $B(T)$ is a group follows from the fact that $B(K_v^{*2})$ is a group. 
The remaining assertions in (2) are a straightforward consequence of Proposition \ref{group extension}.
\end{proof}

\section{Matching valuations}

For any abelian group $\Gamma$, the \it rational rank \rm of $\Gamma$, denoted $\operatorname{rk}_{\mathbb{Q}}(\Gamma)$, is defined to be the dimension of the $\mathbb{Q}$-vector space $\Gamma \otimes_{\mathbb{Z}} \mathbb{Q}$.

We apply Theorem \ref{seeing valuations} to obtain useful results concerning the behaviour of valuations under Witt equivalence; refer to Theorem \ref{main lemma} below. We begin with two lemmas.

\begin{lemma} \label{rank estimate} If $\Gamma$ is a torsion free abelian group and $|\Gamma/2\Gamma| = 2^r$, then $\operatorname{rk}_{\mathbb{Q}}(\Gamma)\ge r$.
\end{lemma}

This is well known. Observe that if $\Gamma \cong \mathbb{Z}\times \dots \times \mathbb{Z}$ ($r$ factors) then $|\Gamma/2\Gamma| = 2^r$, so $\operatorname{rk}_{\mathbb{Q}}(\Gamma)= r$ holds in this case. On the other hand, if $\Gamma = \mathbb{Q}$ for example then $\operatorname{rk}_{\mathbb{Q}}(\Gamma) =1$, $r=0$.

\begin{proof}  We claim that if $\alpha_1, \dots,\alpha_r \in \Gamma$ are such that the cosets $\alpha_i+2\Gamma$, $i=1,\dots,r$ are $\mathbb{F}_2$-linearly independent, then the $\alpha_i$, $i=1,\dots,r$ are $\mathbb{Q}$-linearly independent. Suppose not. Then $\exists$ $k_i \in \mathbb{Z}$ not all zero such that $\sum k_i\alpha_i = 0$. Dividing by a suitable power of $2$, we can assume at least one of the $k_i$ is odd. This contradicts the assumption.
\end{proof}

\begin{lemma} \label{technical lemma} Suppose $v$, $w$ are non-comparable valuations on a field $K$ and $\Gamma_v$ is finitely generated as an abelian group. Then $(1+M_w)K^{*2} \not\subseteq (1+M_v)K^{*2}$.
\end{lemma}

Note: Since the abelian group $\Gamma_v$ is torsion free, the assumption that $\Gamma_v$ is finitely generated is equivalent to $\Gamma_v \cong \mathbb{Z}\times \dots \times \mathbb{Z}$, $r$ times, for some $r\ge 0$.

\begin{proof} Denote by $u$ the finest common coarsening of $v$ and $w$ and by $\overline{v}$ and $\overline{w}$ the valuations on $K_u$ induced by $v$ and $w$ respectively.
Since $\Gamma_{\overline{v}}$ is a subgroup of $\Gamma_v$, $\Gamma_{\overline{v}}$ is also finitely generated.  Replacing $K$ by $K_u$ and $v$ and $w$ by $\overline{v}$ and $\overline{w}$, we are reduced to the case where $v$ and $w$ are independent. Fix $p\in K^*$ with $v(p) \notin 2\Gamma_v$. By the approximation theorem there exists $x\in K$ such that $v(x-p)>v(p)$ and $w(x-1)>0$. Then $x \in 1+M_w$, and $v(x) = v(p) \notin 2\Gamma_v$, so $x\notin U_vK^{*2}$. Since $(1+M_v)K^{*2} \subseteq U_vK^{*2}$ this implies $x \notin (1+M_v)K^{*2}$.
\end{proof}

\begin{theorem} \label{main lemma} Suppose $K$, $L$ are fields, $\alpha : Q(K) \rightarrow Q(L)$ is a hyperfield isomorphism and $v$ is a valuation on $K$ such that $\Gamma_v$ is finitely generated as an abelian group. Suppose either
(i) the basic part of $(1+M_v)K^{*2}$ is $U_vK^{*2}$ and $(1+M_v)K^{*2}$ is unexceptional, or
(ii) the basic part of $(1+M_v)K^{*2}$ is $(1+M_v)K^{*2}$ and $(1+M_v)K^{*2}$ has index 2 in $U_vK^{*2}$.
Then there exists a valuation $w$ on $L$ such that the image of $(1+M_v)K^{*2}/K^{*2}$ under $\alpha$ is $(1+M_w)L^{*2}/L^{*2}$ and $(L^*:U_wL^{*2}) \ge (K^*:U_vK^{*2})$. If (i) holds, then
the image of $U_vK^{*2}/K^{*2}$ under $\alpha$ is $U_wL^{*2}/L^{*2}$.
\end{theorem}

\begin{proof} Let $r := \operatorname{rk}_{\mathbb{Q}}(\Gamma_v)$. If $r=0$ then $v$ is the trivial valuation on $K$, and we take $w$ to be the trivial valuation on $L$ in this case. Assume now that $r>0$.  Set  $T := (1+M_v)K^{*2}$, $S:=\{ s\in L^* : \overline{s}= \alpha(\overline{t}) \text{ for some } t\in T\}$. $U_vK^{*2}$ has index $2^r$. In case (i) $T$ is unexceptional and $B(T) = U_vK^{*2}$, so $B(T)$ has index $2^r$. In case (ii) $B(T)=T$ has index $2^{r+1}$.
The results for $T$ and $B(T)$ carry over to $S$ and $B(S)$ via $\alpha$, i.e., in case (i), $S$ is unexceptional and $B(S)$ is a group of index $2^r$ and, in case (ii), $S$ has index $2^{r+1}$ and $B(S) = S$. Applying Theorem \ref{seeing valuations} to the subgroup $S$ of $L^*$, there exists a valuation $w$ of $L$ with $(1+M_w)L^{*2} \subseteq S$, and $U_wL^{*2}$ has index $\ge 2^r$. In case (i) we can also assume $U_wL^{*2} \subseteq B(S)$. Let $S' = (1+M_w)L^{*2}$, $T' = \{ t\in K^* : \alpha(\overline{t}) = \overline{s} \text{ for some } s \in S'\}$. Note that $B(S') \subseteq U_wL^{*2}$ so the group $B(S')$ has index $\ge 2^r$, and, consequently, the group $B(T')$ has index $\ge 2^r$. If $S' = S$, equivalently, $T'=T$, we are done. Suppose now that $T' \subsetneqq T$ (so, in particular, $T'$ has index $\ge 2^{r+2}$).
Applying Theorem \ref{seeing valuations} one more time, there exists a valuation $v'$ of $K$ with $(1+M_{v'})K^{*2} \subseteq T'$ and $U_{v'}K^{*2}$ has index $\ge 2^r$. (If $T'$ is unexceptional this is clear. If $T'$ is exceptional this is also clear, since then $B(T') = \pm T'$ has index $\ge 2^{r+1}$.)  Then $(1+M_{v'})K^{*2} \subsetneqq (1+M_v)K^{*2}.$ Since $(1+M_{v'})K^{*2} \subseteq (1+M_v)K^{*2}$, $v,v'$ are comparable, by Lemma \ref{technical lemma}.  Since $(1+M_v)K^{*2} \not\subseteq (1+M_{v'})K^{*2}$, $v' \npreceq v$. Consequently, $v \precneqq v'$, so $\Gamma_{v'}$ is a proper quotient of $\Gamma_v$. This contradicts the fact that $U_{v'}K^{*2}$ has index $\ge 2^r$ (so $\operatorname{rk}_{\mathbb{Q}}(\Gamma_{v'})\ge r$).
\end{proof}

\begin{prop} \label{trivial observation} \

(1) Suppose $K$, $L$ are fields and $\alpha : Q(K) \rightarrow Q(L)$ is a hyperfield isomorphism such that the image of $(1+M_v)K^{*2}/K^{*2}$ under $\alpha$ is $(1+M_w)L^{*2}/L^{*2}$. Then $\alpha$ induces a hyperfield isomorphism $K/_m (1+M_v)K^{*2} \rightarrow L/_m (1+M_w)L^{*2}$ such that the obvious diagram
\begin{linenomath}\begin{equation}\label{d1} \xymatrix{
Q(K) \ar[r] \ar[d] & Q(L) \ar[d] \\
K/_m(1+M_v)K^{*2} \ar[r] & L/_m(1+M_w)L^{*2}}\end{equation}\end{linenomath}
commutes.

(2) If, in addition, the image of $U_vK^{*2}/K^{*2}$ under $\alpha$ is $U_wL^{*2}/L^{*2}$, then $\alpha$ induces a hyperfield isomorphism $Q(K_v)\rightarrow Q(L_w)$ and a group isomorphism $\Gamma_v/2\Gamma_v \rightarrow \Gamma_w/2\Gamma_w$ such that the obvious diagrams
\begin{linenomath}\begin{equation}\label{d2} \xymatrix{
K/_m(1+M_v)K^{*2} \ar[r] & L/_m(1+M_w)L^{*2} \\
 Q(K_v) \ar[u] \ar[r]  & Q(L_w) \ar[u]
}\end{equation}\end{linenomath}
and
\begin{linenomath}\begin{equation}\label{d3} \xymatrix{
Q(K)^* \ar[r] \ar[d] & Q(L)^* \ar[d] \\
\Gamma_v/2\Gamma_v \ar[r] & \Gamma_w/2\Gamma_w
}\end{equation}\end{linenomath}
commute. We are assuming here that $v,w$ are non-trivial.
\end{prop}

\begin{proof} (1) Since the image of $(1+M_v)K^{*2}/K^{*2}$ under $\alpha$ is $(1+M_w)L^{*2}/L^{*2}$, $\alpha$ induces a unique bijection $\overline{\alpha} : K/_m (1+M_v)K^{*2} \rightarrow L/_m (1+M_w)L^{*2}$ such that the diagram (\ref{d1}) commutes. Applying Proposition \ref{small detail} (3) one sees that $\overline{\alpha}$ is a hyperfield isomorphism. (2) By our hypothesis the image of $U_vK^{*2}/(1+M_v)K^{*2}$ under $\overline{\alpha}$ is $U_wL^{*2}/(1+M_w)L^{*2}$, so $\overline{\alpha}$ induces a bijection $\alpha' : Q(K_v) \rightarrow Q(L_w)$ such that the diagram (\ref{d2}) commutes. Applying Proposition \ref{group extension} one sees that $\alpha'$ is a hyperfield isomorphism. The last assertion is obvious.
\end{proof}

\section{Abhyankar valuations on function fields over global fields}

Suppose $K$ and $k$ are fields. We say $K$ is a \it function field \rm over $k$ if $K$ is a finitely generated field extension of $k$. If $\operatorname{trdeg}(K:k)=n$ we say $K$ is a \it function field in $n$ variables \rm over $k$. The \it field of constants \rm of $K$ over $k$ (i.e., the algebraic closure of $k$ in $K$) is a finite extension of $k$ \cite[Chapter 10, Proposition 3]{sl}. We do not require that $k$ is the field of constants of $K$ over $k$.
If $K$ is a function field over $k$ and $v$ is a valuation on $K$, the \it Abhyankar inequality \rm asserts that \begin{linenomath}\[\operatorname{trdeg}(K:k) \ge \operatorname{rk}_{\mathbb{Q}}(\Gamma_v/\Gamma_{v|k}) + \operatorname{trdeg}(K_v: k_{v|k}),\]\end{linenomath} where $v|k$ denotes the restriction of $v$ to $k$.  We will say the valuation $v$ is \it Abhyankar \rm (relative to $k$)  if \begin{linenomath}\[\operatorname{trdeg}(K:k) = \operatorname{rk}_{\mathbb{Q}}(\Gamma_v/\Gamma_{v|k}) + \operatorname{trdeg}(K_v: k_{v|k}).\]\end{linenomath}
In this case it is well known that $\Gamma_v/\Gamma_{v|k}$ is finitely generated
and $K_v$ is a function field over $k_{v|k}$. For a proof of these assertions see \cite[Corollary 26]{fvk}.

A \it global field \rm is a field which is either a number field, i.e., a finite extension of $\mathbb{Q}$, or a function field of transcendence degree $1$ over a finite field.

We are interested here in function fields over global fields, equivalently, function fields of transcendence degree $\ge 0$ over $\mathbb{Q}$ or function fields of transcendence degree $\ge 1$ over $\mathbb{F}_p$ for some prime $p$. If $K$ is any field we define the \it nominal transcendence degree \rm of $K$ to be
\begin{linenomath}\[\operatorname{ntd}(K) := \begin{cases} \operatorname{trdeg}(K:\mathbb{Q}) &\text{ if } \operatorname{char}(K) = 0 \\ \operatorname{trdeg}(K:\mathbb{F}_p)-1 &\text{ if } \operatorname{char}(K) = p\ne 0\end{cases}.\]\end{linenomath}
Thus, if $K$ is a function field over a global field $k$, then $\operatorname{ntd}(K) = \operatorname{trdeg}(K:k)$. In this situation, for any valuation $v$ of $K$,

\begin{linenomath}\[\operatorname{rk}_{\mathbb{Q}}(\Gamma_v) := \begin{cases} \operatorname{rk}_{\mathbb{Q}}(\Gamma_v/\Gamma_{v|k}) &\text{ if } v|k \text{ is trivial} \\ \operatorname{rk}_{\mathbb{Q}}(\Gamma_v/\Gamma_{v|k})+1 &\text{ if } v|k \text{ is discrete rank 1}\end{cases},\]\end{linenomath}
and
\begin{linenomath}\[\operatorname{ntd}(K_v) := \begin{cases} \operatorname{trdeg}(K_v:k_{v|k}) &\text{ if } v|k \text{ is trivial} \\ \operatorname{trdeg}(K_v:k_{v|k})-1 &\text{ if } v|k \text{ is discrete rank 1}\end{cases}.\]\end{linenomath}
It follows, for any valuation $v$ of $K$, the Abhyankar inequality implies \begin{linenomath}\[\operatorname{ntd}(K)\ge \operatorname{rk}_{\mathbb{Q}}(\Gamma_v)+\operatorname{ntd}(K_v),\]\end{linenomath} and $v$ is Abhyankar (relative to $k$) iff \begin{linenomath}\[\operatorname{ntd}(K)= \operatorname{rk}_{\mathbb{Q}}(\Gamma_v)+\operatorname{ntd}(K_v).\]\end{linenomath} Moreover, if $v$ is Abhyankar (relative to $k$) then \begin{linenomath}\[\Gamma_v \cong \mathbb{Z}\times \dots \times \mathbb{Z}\]\end{linenomath} (with $\operatorname{rk}_{\mathbb{Q}}(\Gamma_v)$ factors) and $K_v$ is either a function field over a global field (if  $\operatorname{ntd}(K_v)\ge 0$)  or a finite field (if  $\operatorname{ntd}(K_v)=-1$).

\section{Witt equivalence of function fields over global fields}

The main result in this section is Theorem \ref{main theorem 2} which explains how a Witt equivalence of function fields over global fields induces a natural bijection between Abhyankar valuations.

It is important to point out that the bijection between Abhyankar valuations of function fields over global fields is very special. In general, Witt equivalence of two fields does not imply any bijection between valuations whatsoever, as shown in the following simple example:

\begin{example}  Let $F=k((t))$, where $k$ is an algebraically closed field, $\operatorname{char} k \ne 2$. Denote by $v$ the natural valuation on $F$, i.e., 
$$v(\sum_{i=n}^\infty a_it^i) := \min\{ i : a_i \ne 0\} \text{ if } \sum_{i=n}^\infty a_it^i \ne 0.$$ 
The residue field of $(F,v)$ is $k$, the value group is $\mathbb{Z}$. Applying Proposition \ref{group extension}, we see that $Q(F)$ is a group extension of $Q(k) = \{ 0,1\}$ by a cyclic group of order $2$, so $Q(F) = \{ 0,1,p\}$, $p \in Q(F)\backslash Q(k)$, $a+0=a$, $1+1=p+p = \{ 0,1,p\}$, $1+p=\{ 1,p\}$, $a\cdot 0 = 0$, $1\cdot 1 = p\cdot p = 1$, $1\cdot p = p$. It is not difficult to check that exactly the same identities hold true for $Q(\mathbb{F}_5)$, so that $Q(F)\cong Q(\mathbb{F}_5)$ and thus $F \sim \mathbb{F}_5$. At the same time, $F$ has lots of non-trivial valuations, whereas $\mathbb{F}_5$ has only the trivial one.
\end{example}

We begin with some preliminary results.

\begin{lemma} \label{well known} Suppose $K$ is a function field over a global field. Then

(1) There are infinitely many discrete rank one Abhyankar valuations $v$ on $K$.

(2) The group $K^*/K^{*2}$ is infinite.

(3) For any $x\in K^*$, $\exists$ $y\in K^{*2}+xK^{*2}$, $y\notin K^{*2}\cup xK^{*2}$. If $\operatorname{char}(K)\ne 2$ or $x\notin K^{*2}$ one can choose $y\ne 0$.

(4) $B(K^{*2}) = K^*$.
\end{lemma}

 All of this seems to be well-known. Anyway, here is a proof.

\begin{proof} (1) This is clear if $K$ is a number field. Otherwise, $\exists$ a subfield $K_0 \subseteq K$, $\operatorname{ntd}(K_0) = \operatorname{ntd}(K)-1$. Fix $x\in K$ transcendental over $K_0$. $K$ is a finite extension of $K_0(x)$. The principal ideal domain $K_0[x]$ has infinitely many irreducibles. Each irreducible $f$ of $K_0[x]$ defines a discrete rank one valuation $v_f$ on $K_0(x)$ with residue field $K_0[x]/(f)$. The valuation $v_f$  extends in some (possibly non-unique) way to a discrete rank one valuation on $K$ whose residue field is some finite extension of $K_0[x]/(f)$.  (2) is true for any field $K$ having infinitely many inequivalent discrete rank one valuations. Let $v_1,\dots,v_n$ be inequivalent discrete rank one valuations on $K$. Use the approximation theorem to produce $x_i\in K^*$, $i=1,\dots,n$ so that $v_i(x_j)=\delta_{ij}$ (Kronecker's delta), for $i,j=1,\dots,n$. Then the $2^n$ products $x_1^{e_1}\dots x_n^{e_n}$, $e_i \in \{ 0,1\}$, belong to distinct square classes. This proves $|K^*/K^{*2}|\ge 2^n$. Since $n$ is can be chosen to be any positive integer, the result follows. (3) Suppose first that $\operatorname{char}(K) = 2$. If $x\in  K^{*2}$ one can choose $y=0$. If $x\notin K^{*2}$ one can choose $y=1+x$. Suppose now that $\operatorname{char}(K)\ne 2$.
Let $v$ be a discrete rank one Abhyankar valuation on $K$ with $\operatorname{char}(K_v)\ne 2$. Suppose first that $x \in (1+M_v)K^{*2}$, say $x = uc^2$, $u\in 1+M_v$, $c\in K^*$. By induction on the transcendence degree, there exists $\pi(z) \in K_v^*$ and $\pi(d), \pi(e) \in K_v^*$ such that $\pi(z) = \pi(d)^2+\pi(e)^2$, $\pi(z) \notin K_v^{*2}$. Take $y = c^2(d^2+ue^2) =(cd)^2+xe^2$. Then $y\notin K^{*2}\cup xK^{*2}$. If such a valuation $v$ does not exist, then there exist inequivalent discrete rank one valuations $v,w$ on $K$ with $x \notin (1+M_v)K^{*2}$, $x\notin (1+M_w)K^{*2}$. In this case, use the approximation theorem to choose $a \in K^*$ so that $v(a^2)>v(x)$, $w(a^2)<w(x)$.
Define $y= a^2+x$. Then $y = x(1+\frac{a^2}{x}) \in x(1+M_v)K^{*2}$, so $y \notin K^{*2}$. Similarly,
$y= a^2(1+\frac{x}{a^2}) \in (1+M_w)K^{*2}$, so $y\notin xK^{*2}$.
(4) This is immediate from (3).
\end{proof}

\begin{theorem}\label{cases} Suppose $K$ is a function field over a global field and $v$ is an Abhyankar valuation on $K$. Then:

\smallskip

(1) $(K^*: U_vK^{*2}) = 2^{\operatorname{rk}_{\mathbb{Q}}(\Gamma_v)}$.

(2) $(U_vK^{*2}:(1+M_v)K^{*2}) = \begin{cases} \infty &\text{ if } \operatorname{ntd}(K_v) \ge 0 \\ 2 &\text{ if } K_v \text{ is finite, } \operatorname{char}(K_v) \ne 2 \\ 1 &\text{ if } K_v \text{ is finite, } \operatorname{char}(K_v) = 2 \end{cases}.$

\smallskip

(3) The basic part of $T := (1+M_v)K^{*2}$ is

\smallskip

\ \ \ \ $\begin{cases} U_vK^{*2} &\text{ if } \operatorname{ntd}(K_v)\ge 0 \\ \pm T = U_vK^{*2} &\text{ if } K_v \text{ is finite, } \operatorname{char}(K_v) \ne 2, -1 \notin K_v^{*2} \\ T &\text{ if } K_v \text{ is finite, } \operatorname{char}(K_v) \ne 2, -1 \in K_v^{*2} \\ T = U_vK^{*2} &\text{ if } K_v \text{ is finite, } \operatorname{char}(K_v) = 2 \end{cases}.$
\end{theorem}

\begin{proof} (1) is immediate from the isomorphism $K^*/U_vK^{*2} \cong \Gamma_v/2\Gamma_v$. For (2) and (3) one uses the isomorphism  $U_vK^{*2}/(1+M_v)K^{*2} \cong K_v^*/K_v^{*2}$ described in Section 4.  The assertion in (2) in the case $\operatorname{ntd}(K_v)\ge 0$  follows from Lemma \ref{well known} (2) applied to the field $K_v$. The assertions in (2) in the cases where $K_v$ is a finite field are clear. For assertion (3), we apply Proposition \ref{remark on basic part} (2). If $\operatorname{ntd}(K_v)\ge 0$ then $B(K_v^{*2}) = K_v^*$, by Lemma \ref{well known}, so $B(T)=U_vK^{*2}$.
Suppose now that $K_v$ is finite. If $\operatorname{char}(K_v)=2$ then $K_v^* = K_v^{*2}$ so $B(T) = T = U_vK^{*2}$. If $\operatorname{char}(K_v)  \ne 2$, $-1 \notin K_v^{*2}$, then $B(K_v^{*2}) = \pm K_v^{*2} = K_v^*$, so $B(T) = \pm T = U_vK^{*2}$. Finally, if $\operatorname{char}(K_v)  \ne 2$, $-1 \in K_v^{*2}$, then $B(K_v^{*2}) = K_v^{*2}$, so $B(T) = T$.
\end{proof}

\begin{lemma} \label{ntd estimate} Suppose $K$ is a function field over a global field, $L$ is a field, and $\alpha : Q(K) \rightarrow Q(L)$ is a hyperfield isomorphism. Then $\operatorname{ntd}(L) \ge \operatorname{ntd}(K)$.
\end{lemma}

\begin{proof} Let $n := \operatorname{ntd}(K)$. Pick any Abhyankar valuation $v$ on $K$ with $\operatorname{rk}_{\mathbb{Q}}(\Gamma_v) = n$, i.e., $K_v$ is a global field. Choose $w$ as in Theorem \ref{main lemma}. By Proposition \ref{trivial observation}, $K_v \sim L_w$. By Lemma \ref{well known} (2), $K_v^*/K_v^{*2} \cong L_w^*/L_w^{*2}$ is an infinite group, so if $\operatorname{char}(L_w) =p$, $p\ne 0$, then $\operatorname{trdeg}(L_w:\mathbb{F}_p) \ge 1$. By Theorem \ref{main lemma}, $|\Gamma_w/2\Gamma_w| \ge 2^n$, so, by Lemma \ref{rank estimate}, $\operatorname{rk}_{\mathbb{Q}}(\Gamma_w)) \ge n$. The result follows from these two facts and the Abhyankar inequality. In more detail, if $\operatorname{char}(L) = p \ne 0$, then $w$ restricted to $\mathbb{F}_p$ is trivial and $\operatorname{trdeg}(L:\mathbb{F}_p) \ge \operatorname{rk}_{\mathbb{Q}}(\Gamma_w) +\operatorname{trdeg}(L_w: \mathbb{F}_p) \ge n+1$. Similarly, if $\operatorname{char}(L) = 0$, then $\operatorname{trdeg}(L:\mathbb{Q}) \ge n+0 =n$ or $(n-1)+1 =n$, depending on whether $w|_{\mathbb{Q}}$ is trivial or $p$-adic.
\end{proof}

\begin{theorem}\label{main theorem 2} Suppose $K,L$ are function fields over global fields and $\alpha : Q(K) \rightarrow Q(L)$ is a hyperfield isomorphism. Then:

(1) $\operatorname{ntd}(K) = \operatorname{ntd}(L)$.

(2) 
For each Abhyankar valuation $v$ of $K$ with $K_v$ not finite of characteristic $2$ there exists a unique Abhyankar valuation $w$ of $L$ such that $\alpha$ maps $(1+M_v)K^{*2}/K^{*2}$ onto $(1+M_w)L^{*2}/L^{*2}$. $L_w$ is also not finite of characteristic $2$, $\operatorname{rk}_{\mathbb{Q}}(\Gamma_v) = \operatorname{rk}_{\mathbb{Q}}(\Gamma_w)$ and $\operatorname{ntd}(K_v) = \operatorname{ntd}(L_w)$.  

(3)  $\alpha$ maps $U_vK^{*2}/K^{*2}$ onto $U_wL^{*2}/L^{*2}$ except possibly when $K_v$ is finite, $\operatorname{char}(K_v) \ne 2$ and $-1 \in K_v^{*2}$.

(4) For $v,w$ non-trivial, $\alpha$ induces a hyperfield isomorphism $K/_m (1+M_v)K^{*2} \rightarrow L/_m (1+M_w)L^{*2}$ such that diagram (\ref{d1}) commutes. If, in addition, $\alpha$ maps $U_vK^{*2}/K^{*2}$ onto $U_wL^{*2}/L^{*2}$ then $\alpha$ induces a hyperfield isomorphism $Q(K_v) \rightarrow Q(L_w)$ and a group isomorphism $\Gamma_v/2\Gamma_v \rightarrow \Gamma_w/2\Gamma_w$ such that diagrams (\ref{d2}) and (\ref{d3}) commute.

(5) If $v$ corresponds to $w$ and $v'$ corresponds to $w'$ then $v'$ is coarser than $v$ iff $w'$ is coarser than $w$.
\end{theorem}

Note: One can show that $Q(K_v) \cong Q(L_w)$ as hyperfields, and $\Gamma_v/2\Gamma_v \cong \Gamma_w/2\Gamma_w$ as groups, even in the case where $\alpha$ does not map $U_vK^{*2}/K^{*2}$ onto $U_wL^{*2}/L^{*2}$.

\begin{proof} (1) follows from Lemma \ref{ntd estimate} and the symmetry of the hypothesis. (2) Let $n = \operatorname{ntd}(K) = \operatorname{ntd}(L)$. Suppose $w,w'$ are Abhyankar valuations on $L$ and $(1+M_{w'})L^{*2} = (1+M_w)L^{*2}$.
If these groups have infinite index, then the basic parts of these groups are the same, i.e., $U_{w'}L^{*2}= U_wL^{*2}$, i.e., $\operatorname{rk}_{\mathbb{Q}}(\Gamma_{w'})= \operatorname{rk}_{\mathbb{Q}}(\Gamma_{w})$, by Theorem \ref{cases}.
If these groups have finite index then $\operatorname{rk}_{\mathbb{Q}}(\Gamma_{w'}) = n+1 = \operatorname{rk}_{\mathbb{Q}}(\Gamma_w)$, again by Theorem \ref{cases}. Since we already know, by Lemma \ref{technical lemma}, that $w$ and $w'$ are comparable, this proves $w'=w$. This proves the uniqueness of $w$. 
Suppose now that $v$ is an Abhyankar valuation of $K$, $K_v$ not finite of characteristic $2$. The valuation $w$ exists by Theorem \ref{main lemma}. Let $r =\operatorname{rk}_{\mathbb{Q}}(\Gamma_v)$, $s= \operatorname{ntd}(K_v)$.
If $s\ge 0$ then $\alpha$ induces an isomorphism $Q(K_v) \rightarrow Q(L_w)$, by Proposition \ref{trivial observation}, so $\operatorname{ntd}(L_w)\ge s$,  by Lemma \ref{ntd estimate}. If $s=-1$ then $\operatorname{ntd}(L_w)\ge s$ holds trivially.  Also, $|\Gamma_w/2\Gamma_w| \ge |\Gamma_v/2\Gamma_v| = 2^r$, so $\operatorname{rk}_{\mathbb{Q}}(\Gamma_w) \ge r$. Thus \begin{linenomath}\[\operatorname{ntd}(L) \ge \operatorname{rk}_{\mathbb{Q}}(\Gamma_w)+\operatorname{ntd}(L_w) \ge r+s = \operatorname{ntd}(K)= \operatorname{ntd}(L),\]\end{linenomath} so $w$ is Abhyankar, $\operatorname{rk}_{\mathbb{Q}}(\Gamma_w) = r$, and $\operatorname{ntd}(L_w)=s$. In particular, $(U_wL^{*2}:(1+M_w)L^{*2})\ge 2$, so $L_w$ is not finite of characteristic $2$.
This proves (2). (3) and (4) are straightforward.
(5) Suppose now that $v \leftrightarrow w$, $v' \leftrightarrow w'$, $v \preceq v'$. Then $(1+M_{v'})K^{*2} \subseteq (1+M_v)K^{*2}$, so $(1+M_{w'})L^{*2} \subseteq (1+M_w)L^{*2}$.  By Lemma \ref{technical lemma}, $w$ and $w'$ are comparable. If $w' \preceq w$ then $(1+M_w)L^{*2} \subseteq (1+M_{w'})L^{*2}$ so $(1+M_{w'})L^{*2} = (1+M_w)L^{*2}$. We already know that $w=w'$ holds in this case. Thus $w \preceq w'$ holds in any case.
This proves (5).
\end{proof}

The next two lemmas allow one to distinguish the characteristic $2$ case from the characteristic $\ne 2$ case. Denote by $\overline{t} \in K^*/K^{*2}$ the image of $t\in K^*$.

\begin{lemma} \label{characteristic 2} Suppose $K$ is a field, $\operatorname{char}(K)=2$, $\overline{x},\overline{y}\in K^*/K^{*2}$, $\overline{x},\overline{y} \ne 1$ and $\overline{y} \in D_K\langle 1,\overline{x}\rangle$. Then $D_K\langle 1,\overline{y}\rangle = D_K\langle 1,\overline{x}\rangle$.
\end{lemma}

\begin{proof} Suppose
$\overline{z} \in D_K\langle 1,\overline{y}\rangle$.
By our assumptions, $x,y,z \in K^*$, $x,y \notin K^{*2}$, $y=a^2+b^2x$, $z=c^2+d^2y$, $a,b, c,d \in K$. It follows that $z= c^2+d^2(a^2+b^2x) = (c+ad)^2+(bd)^2x$, so $\overline{z} \in D_K\langle 1,\overline{x}\rangle$.  This proves the inclusion $D_K\langle 1,\overline{y} \rangle \subseteq D_K\langle 1,\overline{x}\rangle$. The other inclusion follows from this one,  using the symmetry of the hypothesis (i.e., using $\overline{y}\in D_K\langle 1,\overline{x}\rangle$ $\Leftrightarrow$ $\overline{x} \in D_K\langle 1,\overline{y}\rangle$).
\end{proof}

\begin{lemma} \label{characteristic not 2} Suppose $K$ is a function field over a global field, $\operatorname{char}(K)\ne 2$. Then there exists $\overline{x},\overline{y} \in K^*/K^{*2}$, $\overline{x},\overline{y} \ne 1$ such that $\overline{y}\in D_K\langle 1,\overline{x}\rangle$, $D_K\langle 1,\overline{y}\rangle \not\subseteq D_K\langle 1,\overline{x}\rangle$.
\end{lemma}

\begin{proof}
Fix inequivalent discrete rank one Abhyankar valuations $v,w$ on $K$ with $\operatorname{char}(K_v), \operatorname{char}(K_w)\ne 2$. Choose $x$ so that $v(x)=w(x)=1$ and $a_0,b_0$ so that $w(a_0)=w(b_0)=0$ and the image of $c = a_0^2+b_0^2$ in the residue field of $w$ is not a square. This is possible by Lemma \ref{well known} (3).
Define $y=a^2+x$, $z=b^2+y$ (so $z=a^2+b^2+x$) where $a,b$ are such that $v(a)>0$, $w(a-a_0)> 0$, $w(b-b_0)>0$. Then $v(y) = v(x) = 1$, so $x,y\notin K^{*2}$ and $w(a^2+b^2 -c) >0$ so $z= a^2+b^2+x \in c(1+M_w)$. Thus $x,y,z \in K^*$, $\overline{y} \in D_K\langle 1,\overline{x}\rangle$, $\overline{z} \in D_K\langle 1,\overline{y}\rangle$,  $\overline{x}\ne 1$, $\overline{y} \ne 1$. Let $T =  (1+M_w)K^{*2}$. Thus $T +xT = T\cup xT$ and $z \notin T \cup xT$, so $\overline{z} \notin D_K\langle 1,\overline{x}\rangle$.
\end{proof}

\begin{cor} \label{invariants} Let $K,L$ be function fields over global fields and $K \sim L$. Then

(1) $\operatorname{char}(K)=0$ iff $\operatorname{char}(L) = 0$,

(2) $\operatorname{char}(K)=2$ iff $\operatorname{char}(L) = 2$.
\end{cor}

\begin{proof} For (1), assume $\operatorname{char}(K)=0$. Fix an Abhyankar valuation $v$ of $K$ such that $K_v$ is a number field (so $K_v$ possesses a dyadic valuation). Denote by $w$ the corresponding Abhyankar valuation on $L$. Thus $L_w$ is a global field  and $K_v \sim L_w$. Applying \cite[Theorem 1.1]{Sz1}, $L_w$ also possesses a dyadic valuation so is also a number field. This proves $\operatorname{char}(L)=0$. (2) Assume $\operatorname{char}(K)=2$. Applying Lemma \ref{characteristic 2} and Lemma \ref{characteristic not 2}, we see that $\operatorname{char}(L)=2$. \end{proof}

\begin{rem}

(1) For a global field $K$ the square of the fundamental ideal of its Witt ring of non-singular symmetric bilinear forms vanishes, if $K$ has characteristic 2 (\cite[Theorem III.5.10]{mh}) and does not vanish for global fields of any other characteristic (see \cite[Chapter III]{mh}). Hence, if $K$ and $L$ are Witt equivalent global fields and one field has characteristic 2, the other does also. Corollary \ref{invariants} can be viewed as a certain generalization of this observation.

(2) Any two quadratically closed fields are Witt equivalent, regardless of their characteristics, their Witt ring being just $\mathbb{Z}/2\mathbb{Z}$ (\cite[Proposition 3.1]{l}, \cite[Remark III.3.4]{mh}). Therefore it is, in principle, possible to provide an example of two Witt equivalent fields $K$ and $L$ with $\operatorname{char} K = 2$ and $\operatorname{char} L \neq 2$. However, the authors are not aware of any other examples.

\end{rem}

\begin{lemma} \label{dimension formula}  If $K$ is a function field over a field $k$, $\operatorname{char}(k)=2$, then \begin{linenomath}\[[K:K^2] = 2^{\operatorname{trdeg}(K:k)}\cdot [k:k^2].\]\end{linenomath}
\end{lemma}

\begin{proof} Let $n := \operatorname{trdeg}(K:k)$. Fix $x_1,\dots,x_n$ in $K$ algebraically independent over $k$. Then $K$ is a finite extension of $k(x_1,\dots,x_n)$.  The map $a \mapsto a^2$ defines an isomorphism from $K$ onto $K^2$ which maps $k(x_1,\dots,x_n)$ onto $k^2(x_1^2,\dots,x_n^2)$. It follows that $[K^2: k^2(x_1^2,\dots,x_n^2)] = [K:k(x_1,\dots,x_n)]$. Thus we are reduced to showing that $[k(x_1,\dots,x_n): k^2(x_1^2,\dots,x_n^2)] = 2^n[k:k^2]$. But this is clear.
\end{proof}

\begin{rem} \

(1) It follows from results in \cite{bm} (specifically, from \cite[Theorem 2.9 and Proposition 2.10]{bm}) that (i) if $K,L$ are global fields of characteristic $2$ then $K \sim L$, and (ii) if $K,L$ are function fields over global fields of characteristic $2$ of nominal transcendence degree $1$ or more then $K \sim L$ iff $K \cong L$. One obtains these results by applying Lemma \ref{dimension formula}, taking $k= \mathbb{F}_2$.

(2) For $K,L$ global fields of characteristic $\ne 2$ the meaning of $K \sim L$ is well understood; see for example \cite[Theorem 3.1 and Corollary 3.2]{C}.
\end{rem}

The relationship between Abhyankar valuations $v$ on $K$ with $K_v$ finite, $\operatorname{char}(K_v)=2$ and   Abhyankar valuations $w$ on $L$ with $L_v$ finite, $\operatorname{char}(L_w)=2$ seems to be not very well understood.

\begin{rem} \label{dyadic case} \

(1) If $K$ and $L$ are number fields and $\alpha: Q(K) \rightarrow Q(L)$ is a hyperfield isomorphism the arguments in \cite{Sz1} show that for each dyadic valuation $v$ of $K$ there exists a unique dyadic valuation $w$ of $L$ such that $\alpha$ maps $(1+4M_v)K^{*2}/K^{*2}$ onto $(1+4M_w)L^{*2}/L^{*2}$.

(2) Suppose  $v$  is a dyadic valuation on a number field $K$. Denote by $\tilde{K}_v$ the completion 
of $K$ at $v$. The natural embedding $K \hookrightarrow \tilde{K}_v$ induces a hyperfield isomorphism $K/_m T \cong Q(\tilde{K}_v)$, where $T := (1+4M_v)K^{*2}$. The structure of $Q(\tilde{K}_v)$ is described in \cite[Section 3.6]{M1} for example.

(3) Suppose $K$ is a function field over $\mathbb{Q}$ and $v'$ is an Abhyankar valuation on $K$ such that the residue field $K_{v'}$ is a number field. Suppose also that $v$ is a valuation of $K$ such that $v \preceq v'$ and the induced valuation $\overline{v}$ on $K_{v'}$ is dyadic. Then $M_{v'} \subseteq M_v$ and $4M_{v'} = M_{v'}$ (so $1+M_{v'} \subseteq 1+4M_v$), $K/_m (1+4M_v)K^{*2}$ is a group extension of the hyperfield $K_{v'}/_m(1+4M_{\overline{v}})K_{v'}^{*2}$ in a natural way,
and the following diagram of hyperfields and hyperfield morphisms is commutative:
\begin{linenomath}\begin{equation}\label{d5} \xymatrix{
Q(K) \ar[r] & K/_m(1+M_{v'})K^{*2} \ar[r] & K/_m(1+4M_v)K^{*2} \\ 
& Q(K_{v'}) \ar[u] \ar[r] & K_{v'}/_m (1+4M_{\overline{v}})K_{v'}^{*2} \ar[u]}\end{equation}\end{linenomath}
Here, the horizontal arrows are quotient morphisms and the vertical arrows are group extensions.

(4) It follows from (1), (2) and (3) that if $K,L$ are function fields over global fields and $\alpha: Q(K) \rightarrow Q(L)$ is a hyperfield isomorphism, then there is a well-defined bijection $v \leftrightarrow w$ such that $\alpha$ maps $(1+4M_v)K^{*2}/K^{*2}$ onto $(1+4M_w)L^{*2}/L^{*2}$ between Abhyankar valuations $v$ of $K$ with $K_v$ finite, $\operatorname{char}(K_v)=2$ such that there exists an Abhyankar valuation $v'$ with $v \preceq v'$ and $K_{v'}$ is a number field and  Abhyankar valuations $w$ of $L$ with $L_w$ finite, $\operatorname{char}(L_w)=2$ such that there exists an Abhyankar valuation $w'$ with $w \preceq w'$ and $L_{w'}$ is a number field.  The proof is omitted.
\end{rem}

The relationship between non-Abhyankar valuations $v$ on $K$ and   non-Abhyankar valuations $w$ on $L$ is not very well understood. It is known, by results in \cite{fvk}, that the Abhyankar valuations are dense in the spectral space consisting of all valuations, but this does not seem to help very much.

\section{Further applications}

Let $K$ be a function field in $n$ variables over a global field. For $0\le i \le n$ denote by $\nu_{K,i}$ the set of Abyankar valuations $v$ on $K$ with $\operatorname{ntd}(K_v) = i$. Observe that  \begin{linenomath}\[\nu_{K,i} = \nu_{K,i,0}\cup \nu_{K,i,1} \cup \nu_{K,i,2} \text{ (disjoint union)}\]\end{linenomath} where \begin{linenomath}\[ \nu_{K,i,j} := \begin{cases} \{ v\in \nu_{K,i} : \operatorname{char}(K_v)=0\} \text{ if } j=0 \\ \{ v\in \nu_{K,i} : \operatorname{char}(K_v)\ne 0,2\} \text{ if } j = 1 \\\{ v\in \nu_{K,i} : \operatorname{char}(K_v)=2\} \text{ if } j=2 \end{cases}.\]\end{linenomath}
Of course, some of the sets $\nu_{K,i,j}$ may be empty. Specifically, if $\operatorname{char}(K) = p$ for some odd prime $p$ then $\nu_{K,i,j} = \emptyset$ for $j \in \{ 0,2\}$, and if $\operatorname{char}(K) = 2$ then  $\nu_{K,i,j} = \emptyset$ for $j \in \{ 0,1 \}$.

\begin{cor}\label{one variable} Suppose $K$, $L$ are function fields in $n$ variables over global fields which are Witt equivalent via a hyperfield isomorphism $\alpha : Q(K) \rightarrow Q(L)$. Then for each $i\in \{ 0,1,\dots,n\}$ and each $j\in \{ 0,1,2\}$ there is a uniquely defined bijection  between $\nu_{K,i,j}$ and $\nu_{L,i,j}$ such that, if $v \leftrightarrow w$ under this bijection, then $\alpha$ maps $(1+M_v)K^{*2}/K^{*2}$ onto $(1+M_w)L^{*2}/L^{*2}$ and $U_vK^{*2}/K^{*2}$ onto $U_wL^{*2}/L^{*2}$.
\end{cor}

\begin{proof} The correspondence $v \leftrightarrow w$ is the one defined in Theorem \ref{main theorem 2}. If $v \leftrightarrow w$ then $K_v \sim L_w$ so $v \in \nu_{K,i,j}$ $\Leftrightarrow$ $w \in \nu_{L,i,j}$, for each $i$ and $j$.
\end{proof}

\begin{cor} \label{rational points} Let $K \sim L$ be function fields over number fields,  with fields of constants $k$ and $\ell$ respectively.
If there exists $v \in \nu_{K,0,0}$ with $K_v =k$ and $w \in \nu_{L,0,0}$ with $L_w= \ell$
then $k \sim \ell$.
\end{cor}

\begin{proof} Let $v \leftrightarrow w$ be the bijection between $\nu_{K,0,0}$ and $\nu_{L,0,0}$ defined by Corollary \ref{one variable}. We know that $K_v \sim L_w$ for any $v,w$ related in this way. Since $K_v$ and $L_w$ are number fields, this implies $[K_v:\mathbb{Q}] = [L_w:\mathbb{Q}]$ for any such $v,w$ \cite[Proposition 1.5]{Sz1}. We know also that $k \subseteq K_v$ and $\ell \subseteq L_w$. Choosing $v \leftrightarrow w$ so that $[K_v:\mathbb{Q}] = [L_w:\mathbb{Q}]$ is minimal, we see that $K_v=k$ and $L_w=\ell$.
\end{proof}

\begin{rem} \label{rational points remark} \

(1) Suppose $K$ is the function field of an irreducible $k$-variety which has a non-singular $k$-rational point. (This is always the case, for example, if $K$ is purely transcendental over $k$.) Then there exists $v \in \nu_{K,0,0}$ with $K_v = k$. To prove this one uses the fact that if $A$ is a regular local ring of dimension $n$ with maximal ideal $\frak{m} = (x_1,\dots, x_n)$ and residue field $k$, then $A/(x_n)$ is a regular local ring of dimension $n-1$, and the localization of $A$ at the prime ideal $(x_n)$ is a discrete valuation ring with residue field equal to the field of quotients of $A/(x_n)$; e.g., see \cite[Chapter 11]{am}. Iterating this procedure yields a chain of Abhyankar valuations $v_1 \succeq \dots \succeq v_n$ on $K$ with $\operatorname{trdeg}(K_{v_i}:k) = n-i$, $i=1,\dots,n$ and $K_{v_n} = k$.

(2) If $K$ and $L$ are function fields over global fields of characteristic $\ne 0$, with fields of constants $k$ and $\ell$, respectively, then $K \sim L$ $\Rightarrow$ $k \sim \ell$.
If $k,\ell$ have characteristic $2$ then $[k:k^2] = [\ell:\ell^2] = 2$, by Lemma \ref{dimension formula}, so $k\sim \ell$, by \cite[Proposition 2.10]{bm}. Suppose $k,\ell$ each have characteristic different from $0$ and $2$. Then $k,\ell$ each have level $1$ or $2$. If $k$ has level $1$ then $K$ and consequently also $L$ has level $1$. Since $\ell$ is algebraically closed in $L$ this implies $\ell$ has level $1$. This proves $k$ and $\ell$ have the same level, so $k\sim \ell$, by \cite[Corollary 3.2]{C}.

(3) Combining Corollary \ref{rational points} with (1) and (2) we see that, in particular, \cite[Proposition 3.2]{kop1} is indeed true (even though the proof of \cite[Proposition 3.2]{kop1} given in \cite{kop1} is based on the erroneous argument in \cite[Theorem 1.3]{kop1}).
\end{rem}

Suppose now that $k$ is a number field. Then every ordering of $k$ is archimedean, i.e., corresponds to a real embedding $k \hookrightarrow \mathbb{R}$.
Let $r_1$, respectively $r_2$ be the number of real embeddings of $k$, respectively the number of conjugate pairs of complex embeddings of $k$. Thus $[k:\mathbb{Q}] = r_1+2r_2$. Let \begin{linenomath}\[V_k := \{ r\in k^* : (r) = \frak{a}^2 \text{ for some fractional ideal } \frak{a} \text{ of } k \}.\]\end{linenomath} Here, $(r)$ denotes the fractional ideal of $k$ generated by $r$. Clearly $V_k$  is a subgroup of $k^*$ and $k^{*2} \subseteq V_k$.

\begin{lemma} \label{2 rank} The $2$-rank of
$V_k/k^{*2}$ is $r_1+r_2+2\operatorname{-rk}(C_k)$, where $C_k$ denotes the ideal class group of $k$.
\end{lemma}

\begin{proof} See \cite[Lemma 2.4(a)]{cz}. 
\end{proof}

\begin{lemma} \label{unramified extension} Suppose $K = k(x_1,\dots,x_n)$ and
$v$ is a discrete rank 1 valuation on $k$.  There exists an Abhyankar extension $v'$ of $v$ to $K$ such that $\Gamma_{v'} = \Gamma_v$.
\end{lemma}

\begin{proof}
Define $v'$ by \begin{linenomath}\[ v'(\sum_{\alpha} a_{\alpha}x^{\alpha}) := \min \{ v(a_{\alpha}) : \alpha \in \mathbb{N}^n\} \text{ and } v'(\frac{f}{g}) := v'(f)-v'(g).\]\end{linenomath} Here, $x^{\alpha} := x_1^{\alpha_1}\dots x_n^{\alpha_n}$, for $\alpha \in \mathbb{N}^n$.
\end{proof}

\begin{theorem} \label{genus zero case} Suppose $K = k(x_1,\dots,x_n)$ and  $L = \ell(x_1,\dots, x_n)$ where $n\ge 1$ and $k$ and $\ell$ are number fields,
and $\alpha : Q(K) \rightarrow Q(L)$ is a hyperfield isomorphism. Then

(1) $r \in k^*/k^{*2}$ iff $\alpha(r) \in \ell^*/\ell^{*2}$.

(2) The map $r \mapsto \alpha(r)$ defines a hyperfield isomorphism between $Q(k)$ and $Q(\ell)$.

(3) $\alpha$ maps $V_k/k^{*2}$ to $V_{\ell}/\ell^{*2}$.

(4)The $2$-ranks of the ideal class groups of $k$ and $\ell$ are equal.
\end{theorem}

\begin{proof} Since $k$ is the field of constants of $K$, the canonical group homomorphism from $k^*/k^{*2}$ to $K^*/K^{*2}$ is injective. Claim: The image of the embedding $k^*/k^{*2} \hookrightarrow K^*/K^{*2}$ is equal to $\cap_{v \in \nu_{K,n-1,0}} U_vK^{*2}/K^{*2}$. One inclusion is clear. For the other, use the fact that $D:=k[x_1,\dots,x_n]$ is a UFD.
Suppose $f \in \cap_{v \in \nu_{K,n-1,0}} U_vK^{*2}/K^{*2}$, $f = \frac{g}{h}$, $g,h \in D$, $g,h \ne 0$. Then $f = up_1^{e_1}\dots p_s^{e_s}$, $u \in k^*$, $p_1,\dots, p_s$ irreducibles in $D$. Consider the discrete rank 1 valuation $v_i$ on $K$ associated to $p_i$. Then $v_i \in \nu_{K,n-1,0}$, so $e_i =v_i(f)$ is even, $i=1,\dots,s$. It follows that $p_1^{e_1}\dots p_s^{e_s} \in K^{*2}$ so $f \equiv u \mod K^{*2}$. This proves the claim. Since  the image of $\cap_{v \in \nu_{K,n-1,0}} U_vK^{*2}/K^{*2}$ under $\alpha$ is $\cap_{w \in \nu_{L,n-1,0}} U_wL^{*2}/L^{*2}$, by Corollary \ref{one variable}, this proves (1).
Observe that if $v \leftrightarrow w$, $v\in \nu_{K,0,0}$, $w \in \nu_{L,0,0}$, the diagram
\begin{linenomath}\begin{equation}\label{d4} \xymatrix{
K_v^*/K_v^{*2} \ar[r] &  L_w^*/L_w^{*2} \\
k^*/k^{*2} \ar[u] \ar[r] & \ell^*/\ell^{*2} \ar[u]
}\end{equation}\end{linenomath}
is commutative. The vertical arrows are the maps induced by the field embeddings $k \hookrightarrow K_v$, $\ell \hookrightarrow L_w$. Since the top arrow in diagram (\ref{d4}) defines a hyperfield isomorphism between $Q(K_v)$ and $Q(L_w)$ we know that $[K_v:\mathbb{Q}] = [L_w:\mathbb{Q}]$. Choose $v,w$ with $[K_v:\mathbb{Q}] = [L_w:\mathbb{Q}]$ minimal. The argument in Corollary \ref{rational points} shows that $K_v=k$ and $L_w=\ell$. This proves (2).
Lemma \ref{unramified extension} implies that \begin{linenomath}\[V_k/k^{*2} =\{ r\in k^*/k^{*2} : r \in U_vK^{*2}/K^{*2} \ \forall v \in \nu_{K,n-1,1}\cup \nu_{K,n-1,2}\},\]\end{linenomath} so (3) is clear. Since it is well-known that $r_1$ and $r_2$ are invariant under Witt equivalence, assertion (4) is immediate now, from (3) and Lemma \ref{2 rank}.  This completes the proof.
\end{proof}

\begin{rem} 

(1) The fact that Witt rings of number fields carry some data on the parity of class numbers was first noticed in \cite{Sz2}, and then some additional results were given in \cite{jms}. A deeper study of the relations between Witt equivalence of number fields and 2-ranks of ideal class groups can be found in \cite{cps}.

(2) One can extend Theorem \ref{genus zero case} a bit: Let $V_k^1$ denote the set of all $r \in k^*$ such that $v(r)$ is even for all non-dyadic valuations $v$ of $k$.
By Lemma \ref{unramified extension}, \begin{linenomath}\[V_k^1/k^{*2} = \{ r\in k^*/k^{*2} : r \in U_vK^{*2}/K^{*2} \ \forall v \in \nu_{K,n-1,1} \},\]\end{linenomath} so $\alpha$ maps $V_k^1/k^{*2}$ to $V_{\ell}^1/{\ell}^{*2}$.
Applying this in conjunction with the generalization of Lemma \ref{2 rank} given in  \cite[Lemma 2.4]{ch1} or \cite[Proposition 1]{Sz}, we see that the $S$-class groups of $k$ and $\ell$ have the same $2$-rank,
where 
$S$ consists of all primes which are infinite or dyadic.
\end{rem}

\noindent
\bf Questions: \rm

(1) In Theorem \ref{genus zero case}, is the hypothesis that $K$ and $L$ are purely transcendental over $k$ and $\ell$ really necessary? 

(2) For arbitrary fields $K$ and $L$ is it true that $K(x) \sim L(x)$ $\Rightarrow$ $K \sim L$? 

(3) For fixed integers $n\ge 1$, $m\ge 2$, are there infinitely many Witt inequivalent fields $k(x_1,\dots,x_n)$, $k$ a number field, $[k:\mathbb{Q}] = m$?

\smallskip

Question 3 is interesting because, for given $m$, there are only finitely many Witt inequivalent number fields $k$ with $[k:\mathbb{Q}]=m$. For $m = 1,2,3$ and $4$ these numbers are $1,7, 8$ and $29$  respectively; see \cite{cps} and \cite{jms}.

It is proved in \cite{Sz} that if $\ell$ is a number field, $[\ell:\mathbb{Q}]$ even, and $\ell \ne \mathbb{Q}(\sqrt{-1})$, then, for each integer $t\ge 1$, there exists a number field $k$ such that $k \sim \ell$ and the $2$-rank of the class group of $k$ is $\ge t$. This 
extends an earlier result 
in \cite{cps}.

\begin{cor} \label{infinitely many} For fixed $n\ge 1$ and fixed number field $\ell$, $[\ell:\mathbb{Q}]$ even, $\ell \ne \mathbb{Q}(\sqrt{-1})$, there are infinitely many Witt inequivalent fields of the form $k(x_1,\dots,x_n)$, $k$ a number field, $k \sim \ell$.
\end{cor}

For odd degree extensions Question 3 remains open. Table 2 in \cite{Sz} shows that each of the 8 Witt equivalence classes of cubic extensions contains fields with $2$-rank of the class group equal to $0,1,$ and $2$. Results in \cite{efo} \cite{sch} \cite{s} \cite{wash} show that $0,1,2,3,4, 5,7$ can occur as the $2$-rank of the class group of a cubic field.


\begin{thebibliography}{99}
\bibitem{aej} J.K. Arason, R. Elman, W. Jacob, Rigid elements, valuations, and realization of Witt rings. \textit{J. Algebra} {\bf 110} (1987) 449--467.
\bibitem{ap} J.K. Arason, A. Pfister, Beweis des Krullschen Durchschnittsatzes f\"ur den Wittring. \textit{Invent. Math.} {\bf12} (1971), 173--176.
\bibitem{am} M.F Atiyah, I.G. Macdonald, Introduction to commutative algebra. \textit{Addison-Wesley Publishing Co.}, Reading, Mass.-London-Don Mills, Ont. 1969.
\bibitem{bm} R. Baeza, R. Moresi, On the Witt-equivalence of fields of characteristic 2. \textit{J. Algebra} {\bf92} (1985), no. 2, 446--453.
\bibitem{C} J. Carpenter, Finiteness theorems for forms over global fields.
\textit{Math. Z.} {\bf 209}  no. 1  (1992) 153--166.
\bibitem{ch1} P.E. Conner, J. Hurrelbrink, The 4-rank of $K_2(O)$.
\textit{Canad. J. Math.} {\bf41} (1989), no. 5, 932--960.
\bibitem{cps} P.E. Conner, R. Perlis, K. Szymiczek, Wild sets and 2-ranks of class groups. \textit{Acta Arith.} 79 (1997), no. 1, 83--91.
\bibitem{cz} A. Czoga\l a, On reciprocity equivalence of quadratic number fields. \textit{Acta Arith.} 58 (1991), no. 1, 27--46.
\bibitem{dm2000} M. Dickmann, F. Miraglia, Special groups: Boolean-theoretic methods in the theory of quadratic
forms, \textit{Mem. Amer. Math. Soc.} {\bf145}, American Mathematical Society, Providence, RI, 2000.
\bibitem{e} I. Efrat, Valuations, orderings, and Milnor K-theory. \textit{Mathematical Surveys and Monographs}, {\bf124}. American Mathematical Society, Providence, RI, 2006.
\bibitem{efo} H. Eisenbeis, G. Frey, B. Ommerborn. Computation of the 2-rank of pure cubic fields.
\textit{Math. Comp.} {\bf32} (1978), no. 142, 559--569.
\bibitem{ep} A.J. Engler, A. Prestel, Valued fields.
\textit{Springer Monographs in Mathematics.} Springer-Verlag, Berlin, 2005.
\bibitem{gh} N. Grenier-Boley, D.W. Hoffmann, Isomorphism criteria for Witt rings of real fields. With appendix by Claus Scheiderer.  \textit{Forum Math.} {\bf 25} (2013) 1--18.
\bibitem{h} D.K. Harrison, Witt rings. \textit{University of Kentucky Notes}, Lexington, Kentucky (1970).
\bibitem{j0} W. Jacob, On the structure of Pythagorean fields.
\textit{J. Algebra} {\bf68} (1981), no. 2, 247--267.
\bibitem{j} W. Jacob, Quadratic forms over dyadic valued fields. I. The graded Witt ring.
\textit{Pacific J. Math.} {\bf126}  no. 1, (1987), 21--79.
\bibitem{j1} W. Jacob, Quadratic forms over dyadic valued fields. II. Relative rigidity and Galois cohomology.
\textit{J. Algebra} {\bf148} (1992), no. 1, 162--202.
\bibitem{jms}  S. Jakubec, F. Marko, K. Szymiczek, Parity of class numbers and Witt equivalence of quartic fields.  \textit{Math. Comp.} 64 (1995), no. 212, 1711--1715.
\bibitem{KR} J.L. Kleinstein, A. Rosenberg, Succinct and representational Witt rings. \textit{Pacific J. Math.} {\bf 86} (1980) 99--137.
\bibitem{krw} M. Knebusch, A. Rosenburg, R. Ware, Structure of Witt rings and quotients of Abelian group rings. \textit{Amer. J. Math.} {\bf94} (1972), 119--155.
\bibitem{kop1} P. Koprowski, Local-global principle for Witt equivalence of function fields over global fields. \textit{Colloq. Math.} {\bf 91} (2002) 293--302.
\bibitem{kop2} P. Koprowski, Witt equivalence of algebraic function fields over real closed fields. \textit{Math. Z.} {\bf 242} (2002) 323--345.
\bibitem{kr0} M. Krasner, Approximation des corps valu\'es complets de caract\'eristique $p \ne 0$ par ceux de caract\'eristique 0. 1957 Colloque d'alg\`ebre sup\'erieure, tenu \`a Bruxelles du 19 au 22 d\'ecembre 1956 pp. 129--206
Centre Belge de Recherches Math\'ematiques \'Etablissements Ceuterick, Louvain; Librairie Gauthier-Villars, Paris.
\bibitem{kr} M. Krasner, A class of hyperrings and hyperfields, \textit{Internat. J. Math. and
 Math. Sci.} {\bf 6} (1983) 307--312.
\bibitem{fvk} F.-V. Kuhlmann, On places of algebraic function fields in arbitrary characteristic. \textit{Advances in Math.} {\bf 188} (2004) 399--424.
\bibitem{kss}  M. Kula, L. Szczepanik, K. Szymiczek, Quadratic form schemes and quaternionic schemes.
\textit{Fund. Math.} {\bf130} (1988), no. 3, 181--190.

\bibitem{l} T.-Y. Lam, Introduction to quadratic forms over fields. \textit{Graduate Studies in Mathematics} {\bf 67} American Mathematical Society, Providence, RI, 2005.
\bibitem{sl} S. Lang, Algebra. \textit{Addison-Wesley Publishing Co.}, Reading, Mass. 1965

\bibitem{M1} M. Marshall, Abstract Witt rings, \textit{Queen's Papers in Pure and Applied Math.} {\bf57}, Queen's University, Kingston, Ontario (1980).
\bibitem{M-etc} M. Marshall, The elementary type conjecture in quadratic form theory, \textit{Cont. Math} {\bf 344} (2004),
275--293.
\bibitem{M2} M. Marshall, Real reduced multirings and multifields, \textit{J. Pure and Appl. Alg.} {\bf 205} (2006) 452--468.
\bibitem{mas} Ch.G. Massouros, Methods of constructing hyperfields.
\textit{Internat. J. Math. Math. Sci.} {\bf8} no. 4 (1985) 725--728.
\bibitem{mh} J. Milnor, D. Husemoller, Symmetric bilinear forms.
\textit{Ergebnisse der Mathematik und ihrer Grenzgebiete}, Band {\bf73}. Springer-Verlag, New York-Heidelberg, 1973.
\bibitem{ms} J. Min\'a\v c, M. Spira, Witt rings and Galois groups. \textit{Ann. of Math.} (2) 144 (1996), no. 1, 35--60.
\bibitem{mi} J. Mittas, Sur une classe d'hypergroupes commutatifs.
\textit{C. R. Acad. Sci. Paris S\'er. A-B} {\bf269} 1969 A485--A488.
\bibitem{Petal} R. Perlis, K. Szymiczek, P.E. Conner, R. Litherland, Matching Witts with global fields. \textit{Contemp. Math.} {\bf 155} (1994) 365--378.
\bibitem{r} P. Ribenboim, Th\'eorie des valuations.
Deuxi\`eme \'edition multigraphi\'e. S\'eminaire de Math\'ematiques Sup\'erieures, No. 9 (\'Et\'e, 1964) \textit{Les Presses de l'Universit\'e de Montr\'eal}, Montreal, Que. 1968.
\bibitem{sch} U. Schneiders, Estimating the 2-rank of cubic fields by Selmer groups of elliptic curves.
\textit{J. Number Theory} {\bf62} (1997), no. 2, 375--396.
\bibitem{s} D. Shanks, The simplest cubic fields.
\textit{Math. Comp.} {\bf28} (1974), 1137--1152.
\bibitem{sp1} T.A. Springer, Quadratic forms over fields with a discrete valuation. I. Equivalence classes of definite forms.
\textit{Indag. Math.} {\bf17}, (1955) 352--362.
\bibitem{sp2} T.A. Springer, Quadratic forms over fields with a discrete valuation. II.
\textit{Indag. Math.} {\bf18} (1956) 238--246.
\bibitem{Sz1} K. Szymiczek, Matching Witts locally and globally. \textit{Math. Slovakia} {\bf 41} (1991) 315--330.
\bibitem{Sz2} K. Szymiczek, Witt equivalence of global fields. \textit{Comm. Algebra} {\bf 19} (1991) 1125--1149.
\bibitem{Sz0} K. Szymiczek, Witt equivalence of global fields. II. Relative quadratic extensions. \textit{Trans. Amer. Math. Soc.} {\bf 343} (1994) 277--303.
\bibitem{Sz} K. Szymiczek, 2-ranks of class groups of Witt equivalent number fields. Number theory (Cieszyn, 1998). \textit{Ann. Math. Sil.} No. 12 (1998), 53--64.
\bibitem{wa} R. Ware, Valuation rings and rigid elements in fields.
\textit{Canad. J. Math.} {\bf33} (1981), no. 6, 1338--1355.
\bibitem{wash} L.C. Washington, Class numbers of the simplest cubic fields.
\textit{Math. Comp.} {\bf48} (1987), no. 177, 371--384.
\bibitem{w} E. Witt, Theorie der quadratischen Formen in beliebigen K\"orpern. \textit{Journal f\"ur die reine und angewandte Mathematik} {\bf176} (1937) 31--44.
\end{thebibliography}
\end{document}